\documentclass[a4paper,12pt,reqno]{amsart}
\usepackage[margin=1in]{geometry}
\usepackage[T1]{fontenc}
\usepackage{lmodern}
\usepackage{amsmath,amssymb,mathtools}
\usepackage[hyphens]{url}
\usepackage[colorlinks=true,linkcolor=blue,citecolor=blue,urlcolor=blue]{hyperref}
\usepackage{leftidx}
\usepackage{tikz-cd}
\usepackage{amsthm}

\allowdisplaybreaks
\usepackage{bbm}
\DeclareMathOperator{\GL}{GL}

\DeclareMathOperator{\E}{E}
\DeclareMathOperator{\F}{F}

\DeclareMathOperator{\K}{K}
\DeclareMathOperator{\U}{U}

\DeclareMathOperator{\J}{J}

\DeclareMathOperator{\W}{W}

\DeclareMathOperator{\e}{e}
\renewcommand\v{\operatorname{v}}
\renewcommand\k{\operatorname{k}}
\renewcommand\L{\operatorname{L}}
\DeclareMathOperator{\tr}{tr}
\DeclareMathOperator{\Mat}{Mat}
\DeclareMathOperator{\Ind}{Ind}
\DeclareMathOperator{\ind}{ind}

\DeclareMathOperator{\val}{val}
\DeclareMathOperator{\Gal}{Gal}

\newcommand{\sfrac}{\genfrac{}{}{}1}

\newcommand{\bigslant}[2]{%
  \raisebox{.2em}{$#1$}\left/\raisebox{-.2em}{$#2$}\right.%
}

\newcommand{\bigdoublecoset}[3]{%
  \raisebox{-.2em}{$#1$}
  \left\backslash
  \raisebox{.2em}{$#2$}
  \middle/
  \raisebox{-.2em}{$#3$}
  \right.%
}

\newcommand{\smalldoublecoset}[3]{%
  \raisebox{-.08em}{$\scriptstyle #1$}\!
  \backslash\!
  \raisebox{.08em}{$\scriptstyle #2$}\!
  /\,%
  \raisebox{-.08em}{$\scriptstyle #3$}%
}

\theoremstyle{plain}
\numberwithin{equation}{section}
\newtheorem{theorem}[equation]{Theorem}
\newtheorem*{theorem*}{Theorem}
\newtheorem{lemma}[equation]{Lemma}

\newtheorem{definition}[equation]{Definition}
\newtheorem{proposition}[equation]{Proposition}

\newtheorem{remark}[equation]{Remark}

\title{L\lowercase{ocal} C\lowercase{onverse} T\lowercase{heorems} \lowercase{for} $k$-M\lowercase{iddle} S\lowercase{upercuspidal} R\lowercase{epresentations} \lowercase{of} $p$\lowercase{-adic} $\GL(kn)$}

\author{M\lowercase{oshe} A\lowercase{drian}, M\lowercase{ichael} H\lowercase{orzepa}, \lowercase{and} D\lowercase{avid} C.~L\lowercase{uo}}

\address{Department of Mathematics, Queens College, City University of New York, Flushing, NY 11367, United States}
\email{Moshe.Adrian@qc.cuny.edu}

\address{Mathematics Program, The Graduate Center, City University of New York, New York, NY 10016, United States}
\email{Michael.Horzepa20@gc.cuny.edu}

\address{School of Mathematics, University of Minnesota, Minneapolis, MN 55455, United States}
\email{luo00275@umn.edu}


\begin{document}

\begin{abstract}
Let $\F$ be a non-archimedean local field of characteristic zero with residual characteristic $p$.~In this paper, we study local converse theorems for $k$-middle supercuspidal representations of $\GL(kn, \F)$ via explicit computations of twisted gamma factors on the Galois side, under the assumptions $p \nmid kn$ and $q_{\F}>(k-1)^2$.~Our results provide further evidence for a conjecture of the third-named author and Stevens on depth bounds for the twisting representations in the local converse theorem.~For odd $k \ge 3$, we obtain infinitely many families of supercuspidal representations for which the bound $\left\lfloor \frac{kn}{2} \right\rfloor$ in the local converse theorem is not sharp.
\end{abstract}

\maketitle
\pagestyle{plain}

\section{Introduction}

Let $\F$ be a non-archimedean local field of characteristic zero and residual characteristic $p$, and fix a non-trivial additive quasi-character $\psi_{\F}$ of $\F$.~For irreducible generic representations $\pi$ of $\GL(n,\F)$ and $\tau$ of $\GL(r,\F)$, one may attach an invariant called the \textit{twisted gamma factor} $\gamma(s,\pi\times\tau,\psi_{\F})$ which is a function of the complex variable $s$.~This invariant may be defined via Rankin--Selberg convolution integrals \cite{JPSS} or the Langlands--Shahidi method \cite{Shahidi}.

Fixing $\pi$ and letting $\tau$ range throughout all irreducible generic representations of $\GL(r, \F)$ with $r \geq 1$, one may completely determine $\pi$ via its twisted gamma factors \cite{Chai, JacquetLiu}; indeed, there is an integer~$M_{\pi}$ such that it is enough to consider~$1\le r\le M_{\pi}$, and we call the least such~$M_{\pi}$ the number of "twists" required to characterize~$\pi$.~The question of finding the minimal number of twists needed has a long and rich history (for more information on this, see \cite{ALST}).~It was conjectured that $M_{\pi} \le \left\lfloor \frac{n}{2} \right\rfloor$; this conjecture is commonly referred to as Jacquet's conjecture or the local converse problem for $\GL(n, \F)$ \cite{Jiang}.~In 2018, Jacquet and Liu \cite{JacquetLiu} and, independently, Chai \cite{Chai} proved Jacquet’s conjecture using different methods. 
\begin{theorem*}[Local converse theorem]
Let $n \geq 2$, and let $\pi_1$ and $\pi_2$ be supercuspidal representations of $\GL(n, \, \F)$.~If
\[
\gamma(s, \, \pi_1\times\tau, \, \psi_{\F}) = \gamma(s, \, \pi_2\times\tau, \, \psi_{\F})
\]
as functions of the complex variable $s$, for all supercuspidal representations $\tau$ of $\GL(r, \, \F)$ with
\[
1\leq r\leq\left\lfloor\frac{n}{2}\right\rfloor,
\]
then $\pi_1\cong\pi_2$.
\end{theorem*}

Following the proof of the local converse theorem, the sharpness of the bound $\left\lfloor \frac{n}{2} \right\rfloor$ has been a topic of much curiosity.~While the bound $\left\lfloor \frac{n}{2} \right\rfloor$ is essentially sharp in general \cite{A23, ALST}, there do exist cases in which it is not.~For depth-zero supercuspidals $\pi$, Nien, Zhang, and Yun give precise criteria for when single twists are sufficient, i.e.~$M_{\pi}=1$ \cite{NZ}.~The \textit{simple (epipelagic) supercuspidals}, namely, those of minimal positive depth, also satisfy this single twisting condition \cite[Remark 3.18]{Adrian-Liu}, \cite[Proposition 2.2]{BH-1}, and \cite[Theorem 1.1]{Xu}.

In this paper, we consider a class of depth $\frac{1}{n}$ supercuspidal representations of $\GL(kn, \, \F)$, which we call \textit{$k$-middle supercuspidal representations}.~Here $k$, $n \geq 2$, and we assume that $p \nmid kn$.~These are the supercuspidal representations of depth $\frac{1}{n}$ such that a field extension $\E/\F$ associated to the type from which they are induced has ramification index $n$ and residue degree $k$.~Alternatively, they are the depth $\frac{1}{n}$ supercuspidal representations of $\GL(kn, \, \F)$ which are \textit{minimax} in the sense of \cite{Adrian}.~On the Galois side of the local Langlands correspondence, their Langlands parameters have multiplicity-free restriction to wild inertia and twisting number $k$, i.e.~there are precisely $k$ unramified quasi-characters of $\F^{\times}$ which leave the parameter invariant by twisting.

These representations generalize middle supercuspidals.~Indeed, when $k=2$, the construction recovers the middle supercuspidal representations of $\GL(2n,\F)$ introduced by the third-named author and Stevens \cite{Luo}.~The construction for general $k$ is given in Subsection~\ref{Middle} and does not require the tameness assumption $p\nmid kn$.

Via explicit computations of twisted gamma factors on the Galois side, we show that a $k$-middle supercuspidal representation is determined, among $k$-middle supercuspidals with the same central character, by its twisted gamma factors with tamely ramified quasi-characters of $\F^{\times}$ and simple supercuspidal representations of $\GL(n, \, \F)$, under the additional assumption $q_{\F}>(k-1)^2$.~The first main result we obtain is the following.

\begin{theorem}\label{MainResult1}
Suppose that $p\nmid kn$ and $q_{\F}>(k-1)^2$.~Let $\pi_1$ and $\pi_2$ be $k$-middle supercuspidal representations of $\GL(kn, \, \F)$ with the same central character.~If
\[
\gamma(s, \, \pi_1\times\tau, \, \psi_{\F}) = \gamma(s, \, \pi_2\times\tau, \, \psi_{\F})
\]
as functions of the complex variable $s$ for every $\tau$ that is either a tamely ramified quasi-character of $\F^{\times}$ or a simple supercuspidal representation of $\GL(n, \, \F)$, then $\pi_1\cong\pi_2$.
\end{theorem}

Next, we aim to distinguish $k$-middle supercuspidals from arbitrary supercuspidal representations of $\GL(kn, \, \F)$.~To explain our next main result precisely, let $\mathcal T_{k,n}$ be the union, over all positive proper divisors $d$ of $k$, of $d$-middle supercuspidals of $\GL(dn, \, \F)$.~Here, a $1$-middle means a simple supercuspidal of $\GL(n, \, \F)$.

\begin{theorem}\label{MainResult2}
Suppose that $p\nmid kn$ and $q_{\F}>(k-1)^2$.~Let $\pi$ be a $k$-middle supercuspidal representation of $\GL(kn, \, \F)$ and $\pi'$ be a supercuspidal representation of $\GL(kn, \, \F)$ with the same central character as $\pi$.~If
\[
\gamma(s, \, \pi\times\tau, \, \psi_{\F}) = \gamma(s, \, \pi'\times\tau, \, \psi_{\F})
\]
as functions of the complex variable $s$ for every $\tau$ that is either a tamely ramified quasi-character of $\F^{\times}$ or an element of $\mathcal T_{k,n}$, then $\pi\cong\pi'$.
\end{theorem}

Let $\ell$ be the smallest prime divisor of $k$.~The largest proper divisor of $k$ is $\frac{k}{\ell}$, so the largest degree of a general linear group occurring in $\mathcal T_{k,n}$ is $\frac{kn}{\ell}$.~When $k$ is odd, $\ell\geq3$, and, for $n\geq2$,
\[
\frac{kn}{\ell} < \left\lfloor\frac{kn}{2}\right\rfloor.
\]
Thus, for odd $k\geq3$, Theorem~\ref{MainResult2} gives infinitely many families of supercuspidal representations for which the bound $\left\lfloor\frac{kn}{2}\right\rfloor$ is not sharp.~In particular, when $k$ is an odd prime, the only positive proper divisor of $k$ is $1$, so tamely ramified quasi-characters of $\F^{\times}$ and simple supercuspidal representations of $\GL(n, \, \F)$ suffice to distinguish a $k$-middle from arbitrary supercuspidal representations of $\GL(kn, \, \F)$ with the same central character.

Our results provide further evidence for the conjecture \cite[Conjecture 1.3]{Luo} of the third-named author and Stevens.~This conjecture predicts that, for supercuspidal representations with the same central character, it suffices in the local converse theorem to consider twisting representations of depth at most that of the representations being compared.~Indeed, the twists in Theorem~\ref{MainResult2} all have depth at most $\frac{1}{n}$.

The paper is organized as follows.~In Section~\ref{Notation}, we fix notation and review maximal simple types, admissible pairs, and local factors.~In Section~\ref{Construction}, we recall the structure of simple supercuspidals and give the construction of $k$-middle supercuspidals.~In Section~\ref{GammaComputations}, we compute the relevant twisted gamma factors on the Galois side.~Finally, in Section~\ref{MainProof}, we prove our main results.

\section*{AI Disclosure}

All mathematical ideas in the paper are due to the authors, except for the suggestion to apply Katz's character-sum estimate in the proof of Theorem \ref{MainResult1}, which was provided by ChatGPT.

\section{Notation and background}\label{Notation}

In this section, we first fix notation and recall maximal simple types and the minimal polynomials associated with positive-depth supercuspidal representations.~We then review the parametrization of irreducible $n$-dimensional representations of $\W_{\F}$ in the essentially tame case via admissible pairs, together with the relevant local factors and conductor formulas.~Finally, we state Katz's character-sum estimate, which will be used in the proof of Theorem \ref{MainResult1}.

\subsection{Notation}

Let $\F$ be a non-archimedean local field of characteristic zero with $\mathcal O_{\F}$ its valuation ring, $\k_{\F}$ its residue field with cardinality $q_{\F}$ and characteristic $p$, and $\mu'_{\F}$ the group of roots of unity of order prime to $p$.~Next, let $\mathcal P_{\F}$ denote the maximal ideal of $\mathcal O_{\F}$ and $\varpi_{\F}$ a fixed uniformizer.~Furthermore, let $\psi_{\F}$ be an additive quasi-character of $\F$ with conductor $\mathcal P_{\F}$, i.e.~$\psi_{\F}$ is non-trivial on $\mathcal O_{\F}$ but trivial on $\mathcal P_{\F}$.~Throughout this paper, $k$, $n \geq2$ and $p \nmid kn$.~The hypothesis on $q_{\F}$ is needed only in Section~\ref{MainProof}.

Let $\E/\F$ be a finite extension.~We let $e\left(\E \, | \, \F\right)$ and $f\left(\E \, | \, \F \right)$ denote the ramification index and residue degree of $\E/\F$ respectively.~For an element $x \in \mathcal{O}_{\F}^{\times}$, we denote its reduction in $\k_{\F}^{\times}$ by $\overline{x}$.

Let $\Mat(n\times m, \, \F)$ denote the space of $n\times m$ matrices with entries in $\F$, and let $I_{n}$ denote the $n\times n$ identity matrix.~Lastly, let $\Ind$ denote smooth induction, with compact induction being denoted by $\ind$.~For a supercuspidal representation $\pi$, we write $d(\pi)$ for its depth.

\subsection{Maximal simple types}\label{MaximalSimpleTypes}

From \cite{BK}, we have that supercuspidal representations of $\GL(n, \, \F)$ may be classified in terms of \textit{maximal simple types} which are special pairs $(J, \, \lambda)$ where $J$ is a compact open subgroup of $\GL(n, \, \F)$ and $\lambda$ is an irreducible representation of $J$.~We refer to \cite{BK} for precise definitions and results concerning the objects introduced in this subsection.

Let $V = \F^{n}$ be an $n$-dimensional vector space over $\F$ with standard basis.~Thus we identify $\text{Aut}_{\F}(V)$ with $\GL(n, \, \F)$ and $A = \text{End}_{\F}(V)$ with $\Mat(n \times n, \, \F)$.~Let $\mathfrak{A}$ be a principal hereditary $\mathcal{O}_{\F}$-order in $A$ with Jacobson radical $\mathfrak{P}$.~Define
\[
U^{0}\left(\mathfrak{A}\right) = U\left(\mathfrak{A}\right) = \mathfrak{A}^{\times}, \; U^{m}\left(\mathfrak{A}\right) = I_{n} + \mathfrak{P}^{m}, \; m \geq 1.
\]
For $m \geq 0$, choose $\beta \in A$ such that $\beta \in \mathfrak{P}^{-m} \setminus \mathfrak{P}^{1-m}$ where $\E = \F \left[\beta\right]$ is a field extension of $\F$ and $\E^{\times}$ normalizes $\mathfrak{A}$.~Provided an additional technical condition is satisfied (namely that the critical exponent $k_{0}\left(\beta, \mathfrak{A}\right) < 0$), these data give a principal simple stratum $\left[\mathfrak{A}, m, 0, \beta \right]$ of $A$.~Take $J = J\left(\beta,\mathfrak{A}\right)$, $J^{1} = J^{1}\left(\beta,\mathfrak{A}\right)$, and $H^{1} = H^{1}\left(\beta,\mathfrak{A}\right)$.~Denote by $\mathcal{C}\left(\mathfrak{A}, \, \beta, \, \psi_{\F}\right)$ the set of simple (linear) characters of $H^{1}$.

Recall the following definition of maximal simple types.
\begin{definition}
\normalfont The pair $(J, \, \lambda)$ is called a \textit{maximal simple type} if one of the following holds:
\begin{enumerate}
\item $J = J\left(\beta,\mathfrak{A}\right)$ is an open compact subgroup associated to a simple stratum $\left[\mathfrak{A}, m, 0, \beta \right]$ of $A$ as above, such that, if $\E = \F \left[\beta\right]$ and $B = \text{End}_{\E}(V)$, then $\mathfrak{B} = \mathfrak{A} \cap B$ is a maximal $\mathcal{O}_{\E}$-order in $B$.~Moreover, there exists a simple character $\theta \in \mathcal{C}\left(\mathfrak{A}, \, \beta, \, \psi_{\F}\right)$ such that
    \[
    \lambda \cong \kappa \otimes \sigma
    \]
where $\kappa$ is a $\beta$-extension of the unique irreducible representation $\eta$ of $J^{1} = J^{1}\left(\beta,\mathfrak{A}\right)$ containing $\theta$, and $\sigma$ is the inflation to $J$ of an irreducible cuspidal representation of
    \[
    \bigslant{J}{J^{1}} \cong \bigslant{U\left( \mathfrak{B} \right)}{U^{1}\left( \mathfrak{B} \right)} \cong \GL(r, \, k_{\E}),
    \]
where $r =\frac{n}{[\E \, : \, \F]}$.
    
\item $(J, \, \lambda) = \left(U\left( \mathfrak{A} \right),  \sigma \right)$, where $\mathfrak{A}$ is a maximal hereditary $\mathcal{O}_{\F}$-order in $A$ and $\sigma$ is the inflation to $U\left( \mathfrak{A} \right)$ of an irreducible cuspidal representation of
    \[
    \bigslant{U\left( \mathfrak{A} \right)}{U^{1}\left( \mathfrak{A} \right)} \cong \GL(n, \, \k_{\F}).
    \]
\end{enumerate}
Since the condition $\beta \in \mathfrak{P}^{-m} \setminus \mathfrak{P}^{1-m}$ in case (1) implies that $\beta$ is non-zero, the depth-zero situation in case (2) is not included in the preceding construction.~For notational convenience, we will nevertheless regard case (2) formally as a special case of case (1) by setting $\beta = 0$, $\E = \F$, and $\theta$, $\eta$, $\kappa$ to be trivial.~In either case, we write $\textbf{J} = \E^{\times}J$.~With these data, any supercuspidal representation $\pi$ of $\GL(n, \, \F)$ is of the form
\[
\pi \cong \ind_{\textbf{J}}^{\GL(n, \, \F)}\Lambda
\]
for some choice of $(\textbf{J}, \, \Lambda)$, where $\Lambda \big|_{J} = \lambda$.~We call such a pair $(\textbf{J}, \, \Lambda)$ an \textit{extended maximal simple type}.~The \textit{depth} of $\pi$ is $d(\pi)=\frac{m}{e(\mathfrak A)}=\frac{m}{e(\E|\F)}$ in case (1), and $d(\pi)=0$ in case (2).
\end{definition}
\noindent For $\pi$ a supercuspidal representation of $\GL(n, \, \F)$, any two extended maximal simple types in $\pi$ are conjugate in $\GL(n, \, \F)$.

Finally, we recall that for a positive-depth supercuspidal representation $\pi$, its associated simple stratum $\left[\mathfrak{A}, m, 0, \beta\right]$ has a \textit{characteristic polynomial} $\phi_{Y_{\beta}}(X) \in \k_{\F}[X]$ that depends only on the equivalence class of the final simple stratum in the defining sequence of $\left[\mathfrak{A}, m, 0, \beta\right]$ \cite[Chapter 2, p.~58]{BK}.
\begin{definition}
\normalfont Let $e := e\left(\mathfrak{A}\right)$ and $g := \gcd(e, m)$.~We set
    \[
    Y_{\beta} = \beta^{e/g}\varpi_{\F}^{m/g} \in A.
    \]
Let $\Phi_{Y_{\beta}}(X) \in \F[X]$ be the characteristic polynomial of $Y_{\beta}$ as an $\F$-endomorphism of $V$.~Then, since $Y_{\beta} \in \mathfrak{A}$, we have $\Phi_{Y_{\beta}}(X) \in \mathcal{O}_{\F}[X]$.~We call the reduction $\phi_{Y_{\beta}}(X)$ of $\Phi_{Y_{\beta}}(X)$ modulo $\mathcal{P}_{\F}$ the \textit{characteristic polynomial} of $\left[\mathfrak{A}, m, 0, \beta\right]$.~Furthermore, $\phi_{Y_{\beta}}$ is a power of a monic irreducible polynomial $\varphi_{\pi}(X) \in \k_{\F}[X]$ that we call the \textit{minimal polynomial} of $\left[\mathfrak{A}, m, 0, \beta\right]$.~From \cite[Subsection 6.1]{BHK}, we have that $\varphi_{\pi}$ is an invariant of $\pi$ with $\varphi_{\pi}(X) \neq X$.
\end{definition}

\subsection{Representations of the Weil group and admissible pairs}

In this subsection, we adopt the notation used in \cite[Section 2]{Adrian-Stevens}.~We normalize the additive valuation $\val_{\F}$ on $\F$ to have image $\mathbb{Z} \, \cup \, \{ \infty \}$, and extend this to an additive valuation on a fixed separable closure $\overline{\F}$ of $\F$; thus, for any finite extension $\E/\F$, we have $\val_{\F}(\E^{\times}) = \frac{1}{e(\E|\F)}\mathbb{Z}$.~We also write $| \cdot |$ for the normalized absolute value on $\overline{\F}$, given by
\[
|x| = q_{\F}^{-\val_{\F}(x)}, 
\]
for $x \in \overline{\F}^{\times}$.~Note that if $\E/\F$ is a finite extension, then the usual normalization of the absolute value on $\E$ is given by $x \mapsto |x|^{[\E \, : \, \F]}$, for $x \in \E$.

Next, we define an additive quasi-character $\psi_{\E}$ of $\E$ by
\[
\psi_{\E}=\psi_{\F}\circ \tr_{\E/\F}.
\]
If $\E/\F$ is tamely ramified, then $\psi_{\E}$ is also of level one.~For $r$ a real number, we write
\[
\mathfrak{P}_{\E}^{r} = \{x \in \E \, : \, \val_{\F}(x) \geq r \} \, \, \text{ and } \, \,
\mathfrak{P}_{\E}^{r+} = \{x \in \E \, : \, \val_{\F}(x) > r \}.
\]
From these definitions, $\mathfrak{P}_{\E}^{1/e(\E|\F)} = \mathcal{P}_{\E} = \mathfrak{P}_{\E}^{0+}$.~We put $\U_{\E} = \U_{\E}^{0} = \mathcal{O}_{\E}^{\times}$, with filtration subgroups $\U_{\E}^{r} = 1 + \mathfrak{P}_{\E}^{r}$, for $r >0$, and $\U_{\E}^{r+} = 1 + \mathfrak{P}_{\E}^{r+}$, for $r \geq 0$.

Let $\W_{\F}$ denote the Weil group of $\F$.~Implicit in the definition of the Weil group and local class field theory are compatible choices of isomorphisms $\E^{\times} \cong \W_{\E}^{\text{ab}}$, for each finite extension $\E/\F$, via which we will identify the quasi-characters of $\E^{\times}$ and of $\W_{\E}$.~Moreover, if $\F \subset \E \subset \K$ and $\xi$ is a quasi-character of $\E^{\times}$, let $\xi_{\K} := \xi \circ N_{\K/\E}$ with $N_{\K/\E}$ denoting the field norm map.~When $\xi$ is viewed as a quasi-character of $\W_{\E}$, the quasi-character $\xi_{\K}$ is simply the restriction of $\xi$ to the subgroup $\W_{\K}$.~For $\E/\F$ a finite extension, we write
\[
\Ind_{\E/\F} = \Ind_{\W_{\E}}^{\W_{\F}}
\]
for the induction functor.

In our current situation, there is a nice parametrization of the irreducible $n$-dimensional representations of $\W_{\F}$ in terms of admissible quasi-characters, introduced by Howe \cite{Howe}.
\begin{definition}
\normalfont An \textit{admissible pair of degree $n$} is a pair $(\E/\F, \, \xi)$ where $\E/\F$ is a (tamely ramified) degree $n$ extension and $\xi$ is a quasi-character of $\E^{\times}$ such that
    \begin{enumerate}
\item $\xi$ does not come via the norm from a proper subfield of $\E$ containing $\F$; and
\item if the restriction $\xi \, \big |_{\, \U_{\E}^{0+}}$ comes via the norm from a subfield $\F \subset \text{L} \subset \E$, then $\E/\text{L}$ is unramified.
    \end{enumerate}
Two admissible pairs $(\E_{1}/\F, \, \xi_{1})$ and $(\E_{2}/\F, \, \xi_{2})$ are said to be conjugate if there is an $\F$-isomorphism from $\E_{1}$ to $\E_{2}$ which takes $\xi_{1}$ to $\xi_{2}$.
\end{definition}

In this situation we also say that $\xi$ is an \textit{admissible quasi-character} of $\E^{\times}$ (relative to $\F$).~For any such admissible quasi-character, we write
\[
\rho_{\xi} := \Ind_{\E/\F}\xi.
\]
The following theorem is from \cite[Theorem 2.2.2]{Moy}.
\begin{theorem}\label{TheoremGalois}
If $(\E/\F, \, \xi)$ is an admissible pair of degree $n$, then $\rho_{\xi}$ is an irreducible $n$-dimensional representation of $\W_{\F}$.~Furthermore, two admissible pairs induce equivalent representations if and only if they are conjugate, and each irreducible $n$-dimensional representation of $\W_{\F}$ is induced from an admissible pair.
\end{theorem}
Let $\E/\F$ be a degree $n$ extension, and $\xi$ a \textit{ramified quasi-character} of $\E^{\times}$; that is, $\xi$ is non-trivial on $\U_{\E}$.~Let $d(\xi) \in \frac{1}{e(\E|\F)}\mathbb{Z}$ be the \textit{normalized depth} of $\xi$; that is, $\xi$ is trivial on $\U_{\E}^{d(\xi)+}$ but not on $\U_{\E}^{d(\xi)}$.~There is then $c_{\xi}\in \mathfrak{P}_{\E}^{-d(\xi)}$, well-defined modulo $\mathfrak{P}_{\E}^{-d(\xi)/2}$ such that
\[
\xi(1+x) = \psi_{\E}(c_{\xi}x)
\]
for $x \in \mathfrak{P}_{\E}^{d(\xi)/2+}$.~We say that $c_{\xi}$ \textit{represents} $\xi$.

\subsection{Local factors}\label{LocalFactors}
As in the preceding subsections, let $\E/\F$ be a degree $n$ extension.~Associated with any quasi-character $\xi$ of $\E^{\times}$ are local factors $L(s, \, \xi)$, $\epsilon(s, \, \xi, \, \psi_{\E})$, and $\gamma(s, \, \xi, \, \psi_{\E})$, by Tate’s thesis \cite{Tate} (see also the account in \cite[Section 23]{BH-7}).~When $\xi$ is ramified we have $L(s, \xi) = 1$ and $\epsilon(s, \, \xi, \, \psi_{\E}) = \gamma(s, \, \xi, \, \psi_{\E})$.~Denote by $\epsilon(\xi, \, \psi_{\E})$ the value of the $\epsilon$-factor at $s= 0$.~Note that $\epsilon(s, \, \xi, \, \psi_{\E}) = \epsilon(\xi|\cdot|^{ns}, \, \psi_{\E})$; the extra factor $n$ appears here because of our normalization of the absolute value.~We have the following proposition from \cite[(2.3.17)]{Moy}.
\begin{proposition}\label{PropositionGalois}
For $\xi$ a ramified quasi-character of $\E^{\times}$ of depth $d(\xi)$ with $e(\E \, | \, \F) \, d(\xi)$ odd, we have
    \[
    \epsilon(\xi, \, \psi_{\E}) = \xi^{-1}(c_{\xi}) \, \psi_{\E}(c_{\xi}) \, |c_{\xi}|^{n/2}.~
    \]
\end{proposition}

We recall from \cite[Subsection 29.4]{BH-7}:
\begin{equation}\label{LanglandsConstant}
    \frac{\epsilon \left(s, \, \Ind_{\E/\F}\left(\mathbbm{1}_{\E^{\times}}\right), \, \psi_{\F}\right)}{\epsilon \left(s, \, \mathbbm{1}_{\E^{\times}}, \, \psi_{\E} \right)} = \lambda_{\E/\F}(\psi_{\F})
\end{equation}
where $\mathbbm{1}_{\E^{\times}}$ is the trivial quasi-character of $\W_{\E}$, and $\lambda_{\E/\F}(\psi_{\F})$ is the \textit{Langlands constant} (for more on the Langlands constant, see \cite[Subsection 30.4]{BH-7}).~Let $\F \subset \text{L} \subset \E$.~Then
\begin{equation}\label{LanglandsConstantTower}
    \lambda_{\E/\F}\left(\psi_{\F}\right) = \lambda_{\E/\text{L}}\left(\psi_{\L}\right) \lambda_{\text{L}/\F}\left(\psi_{\F}\right)^{m}
\end{equation}
where $\left[\E : \text{L}\right] = m$ \cite[(2.4.4)]{Moy}.~Furthermore, if $\E/\F$ is a degree $n$ unramified extension, then from \cite[(2.5.3)]{Moy}, we have
\begin{equation}\label{LanglandsConstantUnramified}
    \lambda_{\E/\F}\left(\psi_{\F}\right) = (-1)^{n-1}.
\end{equation}

Next, we define the quasi-character $\varkappa_{\E/\F}$ of $\W_{\F}$ (which we also identify as a quasi-character of $\F^{\times}$ via local class field theory) by
\begin{equation}\label{eqn:varkappa}
\varkappa_{\E/\F} = \det\left(\Ind_{\E/\F}\mathbbm{1}_{\E^{\times}}\right).~
\end{equation}
We further note that $\varkappa_{\E/\F}$ has order dividing two \cite[Subsection 29.2]{BH-7}.

We recall the following result \cite[Lemmas 2.4, 2.5]{ALST}.
\begin{lemma}\label{LemmaGalois}
Let $\E$ and {\normalfont $\text{L}$} be tamely ramified field extensions of $\F$, and let $\chi$ and $\eta$ be quasi-characters of \, $\E^\times$ and \, {\normalfont$\text{L}^\times$}, respectively.~For $g\in \W_{\F}$, write $\E_g=g(\E)$ and {\normalfont $\K_{g}=\E_g\text{L}$}, and set
\[
\theta_g := ({}^{g}\chi)_{ \, \K_{g}}\eta_{ \, \K_{g}}.
\]
Then {\normalfont
\[
\rho_\chi\otimes \rho_\eta
\cong
\bigoplus_{g \, \in \, \smalldoublecoset{\W_{\text{L}}}{\W_{\F}}{\W_{\E}}}
\Ind_{\K_{g}/\F}\theta_g
\]
} and {\normalfont
\[
\gamma(s, \, \rho_\chi\otimes \rho_\eta, \, \psi_{\F}) = \prod_{g \, \in \, \smalldoublecoset{\W_{\text{L}}}{\W_{\F}}{\W_{\E}}}
\lambda_{\K_{g}/\F}(\psi_{\F}) \, \gamma(s, \, \theta_g, \, \psi_{\K_{g}}).
\]
}
\end{lemma}

The local Langlands correspondence preserves the twisted gamma factors considered here \cite{Harris,Henniart1}.~Let $\pi$ be a supercuspidal representation of $\GL(n,\F)$, with central character $\omega_{\pi}$ and Langlands parameter $\Phi_{\pi}$.~From \cite[Proposition~29.2]{BH-7} and \cite[p.~2]{Harris}, we have the identity
\begin{equation}\label{detidentity}
\omega_{\pi} = \det\Phi_{\pi} = \varkappa_{\E/\F}\cdot \xi\,\big|_{\, \F^\times},
\end{equation}
where $\Phi_{\pi}=\Ind_{\E/\F}\xi$ for an admissible pair $(\E/\F,\xi)$ of degree $n$.~Such a pair exists under our standing assumption $p\nmid n$; see \cite[Parametrization Theorem, p.~2]{BH-6}.

For supercuspidal representations $\pi_{1}$ and $\pi_{2}$ of $\GL(N_{1}, \, \F)$ and $\GL(N_{2}, \, \F)$, respectively, with $N_{1} \neq N_{2}$, their Rankin--Selberg $L$-factors and the corresponding dual $L$-factors are $1$ \cite[Proposition 8.1(i)]{JPSS}.~Thus, $\gamma=\epsilon$ for these pairs.

Finally, we introduce the \textit{conductor} $f\left(\pi_{1} \times \pi_{2}, \, \psi_{\F}\right)$ attached to $\pi_{1}$ and $\pi_{2}$.
\begin{theorem}\label{TheoremConductor}\cite[Section 2.7]{JPSS}
    Let $\pi_{1}$ and $\pi_{2}$ be as above.~There are $c \in \mathbb{C}^{\times}$ and $f\left(\pi_{1} \times \pi_{2}, \psi_{\F}\right) \in \mathbb{Z}$ such that 
    \[
    \gamma(s, \, \pi_{1} \times \pi_{2}, \, \psi_{\F}) = c \cdot  q_{\F}^{-f\left(\pi_{1} \times \pi_{2}, \, \psi_{\F}\right)s}.
    \]
\end{theorem}
\noindent We further note that $f\left(\pi_{1} \times \pi_{2}, \psi_{\F}\right) = f\left(\pi_{1} \times \pi_{2}\right)-N_{1}N_{2}$ where $f\left(\pi_{1} \times \pi_{2}\right)$ is independent of $\psi_{\F}$ \cite{BHK}.

\subsection{A character-sum estimate}

We will use the following character-sum estimate of Katz in the proof of Theorem \ref{MainResult1} \cite[Theorem 1]{Katz}.
\begin{theorem}[Katz]\label{KatzEstimate}
Let $\k/\k_{\F}$ be an extension of degree $d\geq2$, let $x$ generate $\k$ over $\k_{\F}$, and let $\theta$ be a non-trivial multiplicative character of $\k^\times$.~Then
\[
\left|\sum_{t\in \k_{\F}}\theta(x-t)\right|\leq(d-1)\sqrt{q_{\F}}.
\]
\end{theorem}

\section{Construction of $k$-Middle Supercuspidal Representations}\label{Construction}

In this section, we first recall the construction of simple supercuspidal representations of $\GL(n, \, \F)$ via maximal simple types.~We then construct $k$-middle supercuspidal representations of $\GL(kn, \, \F)$ and give a parametrization of their isomorphism classes by triples $(\bar{f}_k, \, \chi_k, \, \zeta_k)$.~Finally, we describe the corresponding Langlands parameters, which we use to compute twisted gamma factors in Section~\ref{GammaComputations}.

\subsection{Simple supercuspidal representations}\label{SimpleSupercuspidalRepresentations}

In this subsection, we briefly recall the structure of \textit{simple supercuspidal representations} of $\GL(n, \, \F)$, $n \geq 2$, via the theory of maximal simple types (\cite{Knightly} and \cite[Example 2.18]{Ye}).~Moreover, we carry out the construction here with respect to $\psi_{\F}^{-1}$ rather than $\psi_{\F}$, following the convention in \cite{LuoStevensLLC}.

Let $\mathfrak{I}_{n}$ be the standard minimal hereditary $\mathcal{O}_{\F}$-order in $\mathrm{Mat}(n \times n, \F)$; that is, the space of $n \times n$ matrices with entries in $\mathcal{O}_{\F}$ whose image modulo $\mathcal{P}_{\F}$ is an upper triangular matrix with entries in $\k_{\F}$.~Furthermore, let $\mathfrak{I}_{n}^{+}$ denote the Jacobson radical of $\mathfrak{I}_{n}$; that is, the space of $n \times n$ matrices with entries in $\mathcal{O}_{\F}$ whose image modulo $\mathcal{P}_{\F}$ is a strictly upper triangular matrix with entries in $\k_{\F}$.

Next, let $v \in \mathcal{O}_{\F}^{\times}$ and
\[
\beta_{v} = \begin{pmatrix}
    & & & & (v\varpi_{\F})^{-1} \\
    1 & & & & \\
     & 1 & & & \\
      & & \ddots & & \\
      & & & 1 &
\end{pmatrix}.
\]
Then $\E_{v} = \F \left[\beta_{v} \right]$ is a degree $n$ totally ramified extension of $\F$.~Furthermore, we have that $\text{v}_{\E_{v}}\left(\beta_{v}\right) = -1$ where $\text{v}_{\E_{v}}$ is the normalized discrete valuation on $\E_{v}$.~Let
\begin{equation}\label{MinimalPolynomialSimple}
m_{\beta_{v}}(X) = X^{n} - (v\varpi_{\F})^{-1}
\end{equation}
denote the minimal polynomial of $\beta_{v}$.

Associated with the above are the open compact subgroups of $\mathfrak{I}_{n}^{\times}$:
\[
\U^{1}\left(\mathfrak{I}_{n}\right) := I_{n} + \mathfrak{I}_{n}^{+}
\]
and
\[
\J_{v} := \J\left(\beta_{v}, \mathfrak{I}_{n}\right) = \mathcal{O}_{\F}^{\times}  \U^{1}\left( \mathfrak{I}_{n}\right).
\]
Let $\psi_{\beta_{v}}$ be the quasi-character of $\U^{1}\left(\mathfrak{I}_{n}\right)$ defined by
\[
\psi_{\beta_{v}}(x) = \left(\psi_{\F}^{-1} \circ \tr_{\F}\right)\left(\beta_{v}(x-1)\right)
\]
where $\tr_{\F}$ denotes matrix trace.~Furthermore, let $\varphi$ be a quasi-character of $\k_{\F}^{\times}$ inflated to $\mathcal{O}_{\F}^{\times}$.~From these quasi-characters, we may form a $\beta_{v}$-extension
\[
\kappa_{(v,  \, \varphi)} \, \colon \, \J_{v} \longrightarrow \mathbb{C}^{\times}
\]
of $\psi_{\beta_{v}}$ by letting
\[
\kappa_{(v,  \, \varphi)}(xy) := \varphi(x) \, \psi_{\beta_{v}}(y)
\]
for $x \in \mathcal{O}_{\F}^{\times}$ and $y \in \U^{1}\left(\mathfrak{I}_{n}\right)$.

To form a maximal simple type with the data above, we take the pair $\left(\J_{v},  \kappa_{(v,  \, \varphi)}\right)$.~Moreover, to form an \textit{extended maximal simple type}, let $\textbf{J}_{v} = \E_{v}^{\times}\J_{v}$.~We extend $\kappa_{(v,  \, \varphi)}$ to a quasi-character $\Lambda_{(v,  \, \varphi,  \, \zeta)}$ on $\textbf{J}_{v}$ by setting $\Lambda_{(v,  \, \varphi,  \, \zeta)}\left(\beta_{v}\right) = \zeta \in \mathbb{C}^{\times}$ so that
\[
\Lambda_{(v,  \, \varphi,  \, \zeta)}(\varpi_{\F}) = \left(\zeta^{n}  \, \varphi(v)\right)^{-1}.
\]
We call the pair $\left(\textbf{J}_{v}, \, \Lambda_{(v,  \, \varphi,  \, \zeta)}\right)$ an extended maximal simple type.~Hence, we may now define a simple supercuspidal representation.
\begin{definition}
\normalfont A \textit{simple supercuspidal representation} of $\GL(n, \, \F)$ is a supercuspidal representation of the form
    \[
    \pi_{(v,  \, \varphi,  \, \zeta)} = \ind_{\textbf{J}_{v}}^{\GL(n, \, \F)}\Lambda_{(v,  \, \varphi, \, \zeta)}
    \]
where $\left(\textbf{J}_{v}, \, \Lambda_{(v,  \, \varphi,  \, \zeta)}\right)$ is an extended maximal simple type as above.
\end{definition}
Simple supercuspidal representations are those supercuspidals of $\GL(n, \, \F)$ with minimal positive depth $\frac{1}{n}$.~Let $\omega_{(v,  \, \varphi,  \, \zeta)}$ denote the central character of $\pi_{(v,  \, \varphi,  \, \zeta)}$.

We parametrize simple supercuspidal representations here via $(\bar{v},  \, \varphi,  \, \zeta)$ because there is a bijection between the set of isomorphism classes $\mathcal{A}_{\text{simple}}^{(n)}$ of simple supercuspidal representations of $\GL(n, \, \F)$ and the set of reductions of such triples \cite[Proposition 1.3]{Imai}.
\begin{proposition}\label{BijectionSimple}
There exists a bijection between $\mathcal{A}_{\text{simple}}^{(n)}$ and the set of triples $(\bar{v},  \, \varphi,  \, \zeta)$ with $\bar{v} \in \k_{\F}^{\times}$, $\varphi$ a quasi-character of $\k_{\F}^{\times}$, and $\zeta \in \mathbb{C}^{\times}$.
\end{proposition}

\subsection{Constructing $k$-middle supercuspidal representations}\label{Middle}

In this subsection, we construct depth $\frac{1}{n}$ minimax supercuspidal representations $\pi$ of $\GL(kn, \, \F)$ ($k,n \geq 2$).~Let
\[
f_{k} = X^{k}-a_{k-1}X^{k-1}-\cdots-a_{1}X-a_{0}
\]
with $a_{0},\ldots,a_{k-1} \in \mathcal{O}_{\F}$ such that $\bar{f}_k = f_{k} \pmod{\mathcal{P}_{\F}}$ is irreducible over $\k_{\F}$.~To obtain a matrix form similar to those used for simple and middle supercuspidals, we conjugate the standard hereditary $\mathcal{O}_{\F}$-order corresponding to the partition $kn=\underbrace{k+\cdots+k}_{n\text{ terms}}$ by a suitable element $g\in\GL(kn, \, \F)$.~We choose $g$ so that the resulting hereditary $\mathcal{O}_{\F}$-order $\mathfrak{A}_{kn}$ in $A=\Mat(kn\times kn, \, \F)$ is given by
\begin{equation}\label{Order}
    \mathfrak{A}_{kn}
    =
    \begin{pmatrix}
        \mathfrak{I}_{n}
        & \varpi_{\F}^{-1}\mathfrak{I}_{n}
        & \cdots
        & \varpi_{\F}^{-(k-1)}\mathfrak{I}_{n}
        \\
        \varpi_{\F}\mathfrak{I}_{n}
        & \mathfrak{I}_{n}
        & \cdots
        & \varpi_{\F}^{-(k-2)}\mathfrak{I}_{n}
        \\
        \vdots & \vdots & \ddots & \vdots
        \\
        \varpi_{\F}^{k-1}\mathfrak{I}_{n}
        & \varpi_{\F}^{k-2}\mathfrak{I}_{n}
        & \cdots
        & \mathfrak{I}_{n}
    \end{pmatrix},
\end{equation}
where $\mathfrak{I}_{n}$ is the standard Iwahori order in $\Mat(n \times n, \, \F)$.

Next, let
\begin{equation}
    \beta_{f_{k}}
    :=
    \begin{pmatrix}
        & & & & & & & a_0\varpi_{\F}^{-k} \\
        1 & & & & & & & \\
        & \ddots & & & & & & \vdots \\
        & & 1 & & & & & a_1\varpi_{\F}^{-(k-1)} \\
        & & & \ddots & & & & \vdots \\
        & & & & 1 & & & a_{k-1}\varpi_{\F}^{-1} \\
        & & & & & \ddots & & \\
        & & & & & & 1 &
    \end{pmatrix}
\end{equation}
where $a_j\varpi_{\F}^{-(k-j)} = \left(\beta_{f_{k}}\right)_{jn+1,\,kn}$ for $0\leq j\leq k-1$.~Let
\begin{equation}\label{MinimalPolynomialMiddle}
m_{\beta_{f_{k}}}(X)
= X^{kn}
-a_{k-1}\varpi_{\F}^{-1}X^{(k-1)n}
-\cdots
-a_1\varpi_{\F}^{-(k-1)}X^n
-a_0\varpi_{\F}^{-k}
\end{equation}
denote the minimal polynomial of $\beta_{f_{k}}$ over $\F$.

The element $\beta_{f_{k}}$ gives us a degree-$kn$ extension $\E_{f_{k}} = \F\left[\beta_{f_{k}}\right]$ of $\F$ with ramification index $e(\E_{f_{k}}|\F) = n$ and residue degree $f(\E_{f_{k}}|\F) = k$ such that $\v_{\E_{f_{k}}}\left(\beta_{f_{k}}\right) = -1$ where $\v_{\E_{f_{k}}}$ is the normalized discrete valuation on $\E_{f_{k}}$.~Furthermore, we let $\sigma_{f_{k}} := \beta_{f_{k}}^{n}\varpi_{\F}$ be a zero of $f_{k}$ and $\L_{f_{k}} := \F\left[\sigma_{f_{k}}\right]$ the unique degree $k$ unramified extension of $\F$.~It is not hard to see that $\left[\mathfrak{A}_{kn}, 1, 0, \beta_{f_{k}}\right]$ is a simple stratum with $\beta_{f_{k}}$ minimal over $\F$ \cite[Chapter 1]{BK}.~This implies that the supercuspidal representations coming from our simple stratum will be minimax in the sense of \cite{Adrian}.

Associated with our simple stratum $\left[\mathfrak{A}_{kn}, 1, 0, \beta_{f_{k}}\right]$ is the open compact subgroup $J_{f_{k}} := J\left(\beta_{f_{k}}, \mathfrak{A}_{kn}\right)$ with open normal subgroups $J^{1}:= J^{1}\left(\beta_{f_{k}}, \mathfrak{A}_{kn}\right)$ and $H^{1}:= H^{1}\left(\beta_{f_{k}}, \mathfrak{A}_{kn}\right)$.~In our situation, we have that
\[
J^{1} = H^{1} = U^{1}:= U^{1}\left(\mathfrak{A}_{kn}\right) = I_{kn} + \mathfrak{P}_{kn}
\]
where $\mathfrak{P}_{kn}$ is the Jacobson radical of $\mathfrak{A}_{kn}$:
\begin{equation}\label{Jacobson}
    \mathfrak{P}_{kn}
    =
    \begin{pmatrix}
        \mathfrak{I}_{n}^{+}
        & \varpi_{\F}^{-1}\mathfrak{I}_{n}^{+}
        & \cdots
        & \varpi_{\F}^{-(k-1)}\mathfrak{I}_{n}^{+}
        \\
        \varpi_{\F}\mathfrak{I}_{n}^{+}
        & \mathfrak{I}_{n}^{+}
        & \cdots
        & \varpi_{\F}^{-(k-2)}\mathfrak{I}_{n}^{+}
        \\
        \vdots & \vdots & \ddots & \vdots
        \\
        \varpi_{\F}^{k-1}\mathfrak{I}_{n}^{+}
        & \varpi_{\F}^{k-2}\mathfrak{I}_{n}^{+}
        & \cdots
        & \mathfrak{I}_{n}^{+}
    \end{pmatrix},
\end{equation}
where $\mathfrak{I}_n^{+}$ is the Jacobson radical of $\mathfrak{I}_n$.~Furthermore, $J_{f_{k}} = \mathcal{O}_{\L_{f_{k}}}^{\times} \, U^{1}$.

Next, we note that the set of simple characters $\mathcal{C}\left(\mathfrak{A}_{kn}, \, \beta_{f_{k}}, \, \psi_{\F}\right)$ of $H^{1}$ is the singleton $\{\psi_{\beta_{f_{k}}}\}$ where
\[
\psi_{\beta_{f_{k}}}(x)
=
\left(\psi_{\F} \circ \tr_{\F}\right)
\left(\beta_{f_{k}}(x-1)\right)
\]
for $x \in U^{1}$.~From this simple character, we may define a $\beta_{f_{k}}$-\textit{extension} $\kappa_{(f_{k}, \, \chi_k)}$ \cite[Chapter 5]{BK}.~In our situation, there exist $q_{\F}^{k}-1$ distinct $\beta_{f_{k}}$-extensions which are quasi-characters of $J_{f_{k}}$ given by
\[
   \kappa_{(f_{k}, \, \chi_k)}(xy)
   :=
   \chi_k(x)\psi_{\beta_{f_{k}}}(y)
\]
for $x \in \mathcal{O}_{\L_{f_{k}}}^{\times}$ and $y \in U^{1}$ where $\chi_k$ is a quasi-character of $\mathcal{O}_{\L_{f_{k}}}^{\times}$ that is trivial on $1 + \mathcal{P}_{\L_{f_{k}}}$, i.e.~$\chi_k$ is the inflation of a quasi-character of $\k_{\L_{f_{k}}}^{\times}$ \cite[Theorem 5.2.2]{BK}.~It follows trivially that $\kappa_{(f_{k}, \, \chi_k)}$ satisfies the conditions of $\beta_{f_{k}}$-extensions given in \cite[Chapter 5]{BK}.

We may now introduce simple types $\left(J_{f_{k}}, \, \lambda_{(f_{k}, \, \chi_k)}\right)$ associated with our simple stratum.~In Bushnell--Kutzko notation, we remark that $\lambda_{(f_{k}, \, \chi_k)} = \kappa_{(f_{k}, \, \chi_k)} \otimes \sigma$ may be taken to be another $\beta_{f_{k}}$-extension $\kappa_{(f_{k}, \, \chi_k')}$ as $\sigma$ is a cuspidal representation of $\bigslant{J_{f_{k}}}{J^{1}} \cong \k_{\L_{f_{k}}}^{\times}$ which may be viewed as a quasi-character of $\mathcal{O}_{\L_{f_{k}}}^{\times}$ that is trivial on $1 + \mathcal{P}_{\L_{f_{k}}}$.~This implies that for a given $\beta_{f_{k}}$-extension $\kappa_{(f_{k}, \, \chi_k)}$, the tensor product $\kappa_{(f_{k}, \, \chi_k)} \otimes \sigma$ is another $\beta_{f_{k}}$-extension.~Moreover, we note that such a simple type is maximal as the period $e\left(\mathfrak{A}_{kn}\right)$ of $\mathfrak{A}_{kn}$ is equal to $e(\E_{f_{k}}|\F) = n$.

To get an extended maximal simple type $\left(\textbf{J}_{f_{k}}, \, \Lambda_{(f_{k}, \, \chi_k, \, \zeta_k)}\right)$ from a maximal simple type defined above, we set $\textbf{J}_{f_{k}}:=\E_{f_{k}}^{\times}U^1$ and $\Lambda_{(f_{k}, \, \chi_k, \, \zeta_k)}\left(\beta_{f_{k}}\right) = \zeta_k \in \mathbb{C}^{\times}$ so that
\[
\Lambda_{(f_{k}, \, \chi_k, \, \zeta_k)}(\varpi_{\F})
=
\zeta_k^{-n}\chi_k\left(\sigma_{f_{k}}\right).
\]
We are now ready to introduce the notion of a \textit{$k$-middle supercuspidal representation}.

\begin{definition}\label{MiddleSupercuspidal}
\normalfont We say $\pi_{(f_{k}, \, \chi_k, \, \zeta_k)}$ is a \textit{$k$-middle supercuspidal representation} of $\GL(kn, \, \F)$ if
    \[
    \pi_{(f_{k}, \, \chi_k, \, \zeta_k)}
    =
    \ind_{\textbf{J}_{f_{k}}}^{\GL(kn, \, \F)}
    \Lambda_{(f_{k}, \, \chi_k, \, \zeta_k)}
    \]
where $\left(\textbf{J}_{f_{k}}, \, \Lambda_{(f_{k}, \, \chi_k, \, \zeta_k)}\right)$ is an extended maximal simple type of $\left(J_{f_{k}}, \, \kappa_{(f_{k}, \, \chi_k)}\right)$.~Let $\omega_{(f_{k}, \, \chi_k, \, \zeta_k)}$ denote the central character of $\pi_{(f_{k}, \, \chi_k, \, \zeta_k)}$.
\end{definition}

We may further classify $k$-middle supercuspidals via the following proposition which gives a parametrization of the set $\mathcal{A}_{k\text{-middle}}$ of isomorphism classes of $k$-middle supercuspidals of $\GL(kn, \, \F)$.~
\begin{proposition}\label{Bijection}
There exists a bijection between $\mathcal{A}_{k\text{-middle}}$ and the set of triples $(\bar{f}_k, \, \chi_k, \, \zeta_k)$ with $\bar{f}_k$ a monic irreducible degree $k$ polynomial over $\k_{\F}$, $\chi_k$ a multiplicative quasi-character of $k_{\bar{f}_k}^{\times}$ where $k_{\bar{f}_k} = \bigslant{\k_{\F}[X]}{\left\langle \bar{f}_k(X) \right\rangle}$, and $\zeta_k \in \mathbb{C}^{\times}$.
\end{proposition}
\begin{proof}
We first prove that if $f_{k,1}, f_{k,2} \in \mathcal{O}_{\F}[X]$ are two monic degree $k$ polynomials such that $\bar{f}_{k,i} = f_{k,i} \mod{\mathcal{P}_{\F}}$ is irreducible over $\k_{\F}$ and $\bar{f}_k = \bar{f}_{k,1} = \bar{f}_{k,2}$, then they give the same $k$-middle supercuspidal representations.~Since $\beta_{f_{k,1}}\beta_{f_{k,2}}^{-1} \in U^{1}$, we have that $J_{f_{k,1}} = J_{f_{k,2}}$ and $\textbf{J}_{f_{k,1}} = \textbf{J}_{f_{k,2}}$.~Furthermore, a direct computation shows that $\psi_{\beta_{f_{k,1}}} = \psi_{\beta_{f_{k,2}}}$ as quasi-characters of $U^{1}$ and that $\beta_{f_{k,1}}\beta_{f_{k,2}}^{-1} \in \text{Ker } \psi_{\beta_{f_{k,i}}}$.

Since $\k_{\L_{f_{k,i}}} \cong k_{\bar{f}_k}$ under the isomorphism $\bar{\sigma}_{f_{k,i}}\mapsto X\bmod\bar{f}_k$, we may interpret a quasi-character of $k_{\bar{f}_k}^{\times}$ as a quasi-character of $\k_{\L_{f_{k,i}}}^{\times}$.~This implies $\kappa_{(f_{k,1}, \, \chi_k)} = \kappa_{(f_{k,2}, \, \chi_k)}$.~Furthermore, since $\beta_{f_{k,1}}\beta_{f_{k,2}}^{-1} \in \text{Ker } \psi_{\beta_{f_{k,i}}}$, we have that $\Lambda_{(f_{k,1}, \, \chi_k, \, \zeta_k)} = \Lambda_{(f_{k,2}, \, \chi_k, \, \zeta_k)}$ and $\pi_{(f_{k,1}, \, \chi_k, \, \zeta_k)} = \pi_{(f_{k,2}, \, \chi_k, \, \zeta_k)}$.

The above gives us a well-defined surjective map $(\bar{f}_k, \, \chi_k, \, \zeta_k) \mapsto \pi_{(f_{k}, \, \chi_k, \, \zeta_k)}$ from the set of triples defined in the hypothesis to $\mathcal{A}_{k\text{-middle}}$.~To prove injectivity, suppose
    \[
    \pi_{(f_{k,1}, \, \chi_{k,1}, \, \zeta_{k,1})} \cong \pi_{(f_{k,2}, \, \chi_{k,2}, \, \zeta_{k,2})}.
    \]
We claim
    \[
    (\bar{f}_{k,1}, \, \chi_{k,1}, \, \zeta_{k,1}) = (\bar{f}_{k,2}, \, \chi_{k,2}, \, \zeta_{k,2}).
    \]

From \cite[(2.3.2)]{BK}, we have that the simple stratum $\left[\mathfrak{A}_{kn},1,0,\beta_{f_{k,i}}\right]$ associated with $\pi_{(f_{k,i}, \, \chi_{k,i}, \, \zeta_{k,i})}$ is fundamental.~Hence, we may set
    \[
    Y_{\beta_{f_{k,i}}}
    =
    \beta_{f_{k,i}}^{n}\varpi_{\F}
    \in A.
    \]
Reducing the characteristic polynomial $\Phi_{Y_{\beta_{f_{k,i}}}}(X)$ of $Y_{\beta_{f_{k,i}}}$ modulo $\mathcal{P}_{\F}$, we obtain $\bar{f}_{k,i}^{\,n}$.~Taking its unique monic irreducible factor, we have
    \[
    \varphi_{\pi_{(f_{k,i}, \, \chi_{k,i}, \, \zeta_{k,i})}}(X) = \bar{f}_{k,i}.
    \]
This implies $\bar{f}_k = \bar{f}_{k,1} = \bar{f}_{k,2}$ because the minimal polynomial is an invariant of $\pi_{(f_{k,i}, \, \chi_{k,i}, \, \zeta_{k,i})}$ \cite[Subsection 6.1]{BHK}.~Next, we fix a lift $f_{k} \in \mathcal{O}_{\F}[X]$ of $\bar{f}_k$.

From \cite[(6.2.4)]{BK}, we have that there exists $g \in \GL(kn, \, \F)$ such that
    \[
    \Lambda_{(f_{k}, \, \chi_{k,1}, \, \zeta_{k,1})}^{g} = \Lambda_{(f_{k}, \, \chi_{k,2}, \, \zeta_{k,2})};
    \]
this implies $\leftidx{^{g}}{\psi_{\beta_{f_{k}}}}{} = \psi_{\beta_{f_{k}}}$.~Hence, \cite[Corollary 3.3.17]{BK} tells us that $g$ is contained inside the normalizer $N_G(\psi_{\beta_{f_{k}}}) = \textbf{J}_{f_{k}}$ of $\psi_{\beta_{f_{k}}}$.~Therefore,
    \[
    \Lambda_{(f_{k}, \, \chi_{k,1}, \, \zeta_{k,1})} = \Lambda_{(f_{k}, \, \chi_{k,2}, \, \zeta_{k,2})},
    \]
which means $\zeta_{k,1} = \zeta_{k,2}$ and $\chi_{k,1} = \chi_{k,2}$, thus completing the proof of injectivity.
\end{proof}

\subsection{Langlands parameters}

Throughout the rest of this paper, let
\[
\xi := \xi_{(f_{k}, \, \chi_k, \, \zeta_k)},
\]
where $(\E_{f_{k}}/\F, \, \xi)$ is an admissible pair of degree $kn$ corresponding to the $k$-middle supercuspidal representation $\pi_{(f_{k}, \, \chi_k, \, \zeta_k)}$ of $\GL(kn, \, \F)$ \cite[Parametrization Theorem, p.~2]{BH-6}.~Then the Langlands parameter of $\pi_{(f_{k}, \, \chi_k, \, \zeta_k)}$ is
\[
\Phi_{(f_{k}, \, \chi_k, \, \zeta_k)}
=
\Ind_{\E_{f_{k}}/\F}\xi.
\]

The following proposition describes the restriction of $\xi$ to $1+\mathcal{P}_{\E_{f_{k}}}$, and this restriction in turn determines the representing element of $\xi$.~We note that $\E_{f_{k}}$ is naturally embedded in $\Mat(kn\times kn, \, \F)$.

\begin{proposition}\label{1units}
On $1+\mathcal{P}_{\E_{f_{k}}}$, the quasi-character $\xi$ satisfies
\[
\xi \, \big|_{\, 1+\mathcal{P}_{\E_{f_{k}}}}
=
\psi_{\beta_{f_{k}}} \, \big|_{\, 1+\mathcal{P}_{\E_{f_{k}}}}.
\]
\end{proposition}

\begin{proof}
We first consider the \textit{admissible 1-pair}
\[
\left(
\E_{f_{k}}/\F, \,
\xi \, \big|_{ \, 1+\mathcal{P}_{\E_{f_{k}}}}
\right)
\]
as defined in \cite[Subsection 1.3]{BH-4}.~Associated with $\pi_{(f_{k}, \, \chi_k, \, \zeta_k)}$ is the simple character $\psi_{\beta_{f_{k}}}$ whose \textit{endo-class} we denote by $\Theta_{\beta_{f_{k}}}$ (for an in-depth definition of the endo-class of a simple character, see \cite{BH-2}).~Since $p\nmid kn$, we have that $\Theta_{\beta_{f_{k}}}$ is \textit{tamely ramified} \cite{BH-2}.~Hence, the discussion following \cite[Theorem 1.3]{BH-4} and \cite[Proposition 2.3]{BH-4} imply Proposition~\ref{1units}.
\end{proof}

Likewise, write
\[
\eta := \eta_{(v, \, \varphi, \, \zeta')},
\]
where $(\E_v/\F, \, \eta)$ is an admissible pair of degree $n$ corresponding to the simple supercuspidal representation $\pi_{(v, \, \varphi, \, \zeta')}$ of $\GL(n, \, \F)$.~Then its Langlands parameter is
\[
\Phi_{(v, \, \varphi, \, \zeta')}
=
\Ind_{\E_v/\F}\eta,
\]
and $\eta$ is represented by $-\beta_{v}$ \cite[p.~129]{Adrian-Liu}.~Explicitly, $\eta$ is given by the following proposition, using our convention that the simple supercuspidal is constructed with respect to $\psi_{\F}^{-1}$ \cite[Theorem 3.17]{Adrian-Liu}.

\begin{proposition}\label{SimpleSupercuspidalParameter}
The quasi-character $\eta$ is uniquely determined by the following:
\begin{enumerate}
\item for $y\in 1+\mathcal{P}_{\E_v}$,
    \[
    \eta(y)
    = (\psi_{\beta_{v}}\circ\iota)(y),
    \]
where
    \[
    \iota \, : \, \E_v^\times \hookrightarrow \GL(n, \, \F)
    \]
is the embedding induced by a chosen $\F$-basis of $\E_v$;

\item on $\k_{\F}^{\times}$,
    \[
    \eta \, \big|_{\, \k_{\F}^{\times}} = \omega_{(v, \, \varphi, \, \zeta')} \, \big|_{\, \k_{\F}^{\times}} \otimes \left(\varkappa_{\E_v/\F} \, \big|_{\, \k_{\F}^{\times}}\right)^{-1};
    \]

\item at $\beta_{v}^{-1}$,
    \[
    \eta(\beta_{v}^{-1}) = \left( \zeta'\cdot\lambda_{\E_v/\F}(\psi_{\F}^{-1}) \right)^{-1}.
    \]
\end{enumerate}
\end{proposition}

\section{Computation of Twisted Gamma Factors}\label{GammaComputations}

In this section, we compute the twisted gamma factors of $k$-middle supercuspidal representations on the Galois side.~We first consider twists by tamely ramified quasi-characters of $\F^\times$, and then twists by simple supercuspidal representations of $\GL(n, \, \F)$.~Throughout this section, we assume that $\F$ has characteristic zero and impose the condition that $p\nmid kn$.

\subsection{Twisting by tamely ramified quasi-characters}

Let $\nu$ be a tamely ramified quasi-character of $\F^\times$.~We note that
\[
\left(\Ind_{\E_{f_{k}}/\F}\xi\right)\otimes\nu = \Ind_{\E_{f_{k}}/\F} \left(\xi\otimes\nu_{\E_{f_{k}}}\right)
\]
where $\nu_{\E_{f_{k}}}=\nu\circ N_{\E_{f_{k}}/\F}$.~Hence,
\begin{equation}\label{Equation1}
\gamma(s, \, \Phi_{(f_{k}, \, \chi_k, \, \zeta_k)} \otimes \nu, \, \psi_{\F})
= \gamma\left( s, \, \Ind_{\E_{f_{k}}/\F} \left(\xi\otimes\nu_{\E_{f_{k}}}\right), \, \psi_{\F} \right).
\end{equation}
Let
\[
B_{f_{k}} := \lambda_{\E_{f_{k}}/\F}(\psi_{\F})\xi(\beta_{f_{k}})^{-1}.
\]
We have the following proposition.
\begin{proposition}\label{PropositionGaloisTamelyRamified}
For a tamely ramified quasi-character $\nu$ of $\F^\times$, we have
\[
\gamma(s, \, \Phi_{(f_{k}, \, \chi_k, \, \zeta_k)} \otimes\nu, \, \psi_{\F}) = B_{f_{k}} \, \nu\left((-1)^{kn+1}a_0^{-1}\varpi_{\F}^{k}\right) q_{\F}^{k(1/2-s)}.
\]
\end{proposition}

\begin{proof}
Since $\xi\otimes\nu_{\E_{f_{k}}}$ is ramified, we have from \eqref{LanglandsConstant} and \eqref{Equation1} that
\begin{align*}
\gamma(s, \, \Phi_{(f_{k}, \, \chi_k, \, \zeta_k)} \otimes\nu, \, \psi_{\F}) & = \gamma\left(s, \, \Ind_{\E_{f_{k}}/\F} \left(\xi\otimes\nu_{\E_{f_{k}}}\right), \, \psi_{\F} \right)\\
& = \lambda_{\E_{f_{k}}/\F}(\psi_{\F})\, \gamma(s, \, \xi\otimes\nu_{\E_{f_{k}}}, \, \psi_{\E_{f_{k}}}).
\end{align*}
Next, Proposition~\ref{1units} tells us that the representing element $c_{\xi}$ of $\xi$ is $\beta_{f_{k}}$ and $\val_{\F}(\beta_{f_{k}})=-\sfrac{1}{n}$.~Since $\nu_{\E_{f_{k}}}$ is tamely ramified, $\xi\otimes\nu_{\E_{f_{k}}}$ has the same representing element and depth as $\xi$.

We have
\[
N_{\E_{f_{k}}/\F}(\beta_{f_{k}}) = (-1)^{kn+1}a_0\varpi_{\F}^{-k}
\]
and
\[
\tr_{\E_{f_{k}}/\F}(\beta_{f_{k}}) = 0.
\]
Since $\e(\E_{f_{k}}\mid\F)d(\xi) = 1$, we have from Proposition~\ref{PropositionGalois} that
\begin{align*}
\gamma(s, \, \xi\otimes\nu_{\E_{f_{k}}}, \, \psi_{\E_{f_{k}}}) & = \xi(\beta_{f_{k}}^{-1}) \nu_{\E_{f_{k}}}(\beta_{f_{k}}^{-1}) |\beta_{f_{k}}|^{kn(1/2-s)} \\
&= \xi(\beta_{f_{k}}^{-1}) \nu\left((-1)^{kn+1}a_0^{-1}\varpi_{\F}^{k}\right)q_{\F}^{k(1/2-s)}.
\end{align*}
Proposition~\ref{PropositionGaloisTamelyRamified} now follows.
\end{proof}

\subsection{Twisting by simple supercuspidal representations}
\label{SimpleSupercuspidalGalois}

For $g\in\W_{\F}$, let
\begin{equation}\label{CompositumDegree}
\K_{g} := g(\E_{f_{k}})\E_v
\, \, \text{and}\, \, 
m_g:=[\K_{g}:g(\E_{f_{k}})].
\end{equation}
Then
\[
[\K_{g} : \F] = k n m_g.
\]
From \cite[Theorem 3, p.~504]{Cornell}, we have the following result:

\begin{theorem}[Abhyankar's lemma]\label{Abhyankar}
Let $\E_1$, $\E_2$, and $\F$ be local fields, $\E_1$, $\E_2$ finite extensions of $\F$.~Suppose $\E_2/\F$ is tamely ramified and $e(\E_2\mid\F)\mid e(\E_1\mid\F)$.~Then $\E_2\E_1$ is an unramified extension of $\E_1$.
\end{theorem}

Theorem~\ref{Abhyankar} gives rise to the following diagram that illustrates the ramification indices and residue degrees between $\K_{g}$, $g(\E_{f_{k}})$, $\E_v$, and $\F$:
\begin{equation}\label{Diagram}
\begin{tikzpicture}[scale=1.3]
  \node (Kg) at (0,2) {$\K_{g}$};
  \node (Ef) at (-2,0) {$g\left(\E_{f_{k}}\right)$};
  \node (Ev) at (2,0) {$\E_v$};
  \node (F) at (0,-2) {$\F$};

  \draw (Kg) -- node[left, sloped, above]
    {$e=1,\ f=m_g$} (Ef);
  \draw (Kg) -- node[right, sloped, above]
    {$e=1,\ f=km_g$} (Ev);

  \draw (Ef) -- node[pos=0.42, left, sloped, below]
    {$e=n,\ f=k$} (F);
  \draw (Ev) -- node[pos=0.42, right, sloped, below]
    {$e=n,\ f=1$} (F);

  \draw (Kg) -- (F)
    node[pos=0.50, fill=white, inner sep=3pt]
    {$\begin{array}{c}
        e=n\\[3pt]
        f=km_g
      \end{array}$};
\end{tikzpicture}.
\end{equation}

Let
\[
\theta_g:=
\left({}^{g}\xi\circ N_{\K_{g}/g(\E_{f_{k}})}\right)
\cdot
\left(\eta\circ N_{\K_{g}/\E_v}\right).
\]
We first compute the representing element $c_{\theta_g}$ of $\theta_g$.

\begin{proposition}\label{RepresentingElement}
The quasi-character $\theta_g$ has depth $d(\theta_g)=\frac{1}{n}$ and is represented by
\[
c_{\theta_g} := g(\beta_{f_{k}})-\beta_{v}.
\]
\end{proposition}

\begin{proof}
Since $\K_{g}/\E_v$ is tame, $\eta\circ N_{\K_{g}/\E_v}$ is represented by $-\beta_{v}$ and has depth $\sfrac{1}{n}$ \cite[Subsection 2.1]{Adrian-Stevens}.~Likewise, we have that ${}^{g}\xi\circ N_{\K_{g}/g(\E_{f_{k}})}$ has the same depth and is represented by $g(\beta_{f_{k}})$ \cite[Subsection 2.1]{Adrian-Stevens}.

For $x\in\mathfrak{P}_{\K_{g}}^{\frac{1}{n}}=\mathcal{P}_{\K_{g}}$, we have
\[
\left({}^{g}\xi\circ N_{\K_{g}/g(\E_{f_{k}})}\right)(1+x) = \psi_{\K_{g}}\left(g(\beta_{f_{k}})x\right)
\]
and
\[
\left(\eta\circ N_{\K_{g}/\E_v}\right)(1+x) = \psi_{\K_{g}}(-\beta_{v}x).
\]
Since $\val_{\F}(\beta_{v})=\val_{\F}(g(\beta_{f_{k}}))=-\sfrac{1}{n}$, the quotient $u:=-\frac{\beta_{v}}{g(\beta_{f_{k}})}$ is a unit of $\mathcal{O}_{\K_{g}}$.~Therefore,
\[
\theta_g(1+x) = \psi_{\K_{g}}\left((g(\beta_{f_{k}})-\beta_{v})x\right) = \psi_{\K_{g}}\left((1+u)g(\beta_{f_{k}})x\right).
\]

It remains to verify that $1+u\in\mathcal{O}_{\K_{g}}^{\times}$.~For the sake of contradiction, suppose $1+u\in\mathcal{P}_{\K_{g}}$.~Then $\overline{u}=-1$ in $k_{\K_{g}}$.~Raising both sides of $ug(\beta_{f_{k}})=-\beta_{v}$ to the power of $n$, we have that
\begin{equation}\label{Relation1}
u^ng(\sigma_{f_{k}})=(-1)^nv^{-1}.
\end{equation}
Reducing \eqref{Relation1} modulo $\mathcal{P}_{\K_{g}}$, we have that $\overline{g(\sigma_{f_{k}})}=\overline{v}^{-1}\in \k_{\F}^{\times}$.

Thus we obtain a contradiction:~$\overline{\sigma_{f_{k}}}$ generates an extension of degree $k$ of $\k_{\F}$, and hence so does $\overline{g(\sigma_{f_{k}})}$.~Therefore, we have that $\val_{\F}(g(\beta_{f_{k}})-\beta_{v})=-\sfrac{1}{n}$ and $d(\theta_g)=\sfrac{1}{n}$.
\end{proof}

Next, we compute the twisted gamma factor $\gamma\left(s, \, \Phi_{(f_{k}, \, \chi_k, \, \zeta_k)} \otimes\Phi_{(v, \, \varphi, \, \zeta')}, \, \psi_{\F}\right)$.~For the following proposition, let
\begin{equation}\label{D(v)}
D(v):= \left|\, \bigdoublecoset{\W_{\E_v}}{\W_{\F}}{\W_{\E_{f_{k}}}} \,\right|
\end{equation}
where $v\in\mathcal{O}_{\F}^{\times}$.

\begin{proposition}\label{GaloisPropositionMain}
Let $\Phi_{(f_{k}, \, \chi_k, \, \zeta_k)}$ be the Langlands parameter of $\pi_{(f_{k}, \, \chi_k, \, \zeta_k)}$ and $\Phi_{(v, \, \varphi, \, \zeta')}$ the Langlands parameter of $\pi_{(v, \, \varphi, \, \zeta')}$.~Then
\begin{align}
&\gamma(s, \, \Phi_{(f_{k}, \, \chi_k, \, \zeta_k)} \otimes\Phi_{(v, \, \varphi, \, \zeta')}, \, \psi_{\F})
\notag \\
&= (-1)^{n-D(v)}B_{f_{k}}^{n}\, \xi\left(\frac{\sigma_{f_{k}}}{\sigma_{f_{k}}-v^{-1}}\right) \varkappa_{\E_v/\F}
\left(\frac{(-1)^{kn}\varpi_{\F}^{k}}{f_{k}(v^{-1})}\right) \omega_{(v, \, \varphi, \, \zeta')} \left(\frac{(-1)^{kn}\varpi_{\F}^{k}}{f_{k}(v^{-1})}\right) q_{\F}^{kn(1/2-s)}.
\label{simpleformula}
\end{align}
\end{proposition}

\begin{proof}
From Lemma~\ref{LemmaGalois} and Propositions~\ref{RepresentingElement} and~\ref{PropositionGalois}, we have that
\begin{equation}\label{GammaProduct}
\gamma\left(s, \, \Phi_{(f_{k}, \, \chi_k, \, \zeta_k)}
\otimes\Phi_{(v, \, \varphi, \, \zeta')}, \, \psi_{\F}\right) = \prod_{\smalldoublecoset{\W_{\E_v}}{\W_{\F}}{\W_{\E_{f_{k}}}}}
\lambda_{\K_{g_{i}}/\F}(\psi_{\F})
\gamma\left(s, \, \theta_{g_{i}}, \, \psi_{\K_{g_{i}}}\right)
\end{equation}
where $\{g_{i}\}$ is a set of representatives of the double coset space $\W_{\E_v}\backslash\W_{\F}/\W_{\E_{f_{k}}}$.~For each $\gamma\left(s,\,\theta_{g_i},\,\psi_{\K_{g_i}}\right)$,
Proposition~\ref{PropositionGalois} shows that the factor is of the form
\begin{equation}\label{GammaFactorTerm}
    \theta_{g_i}^{-1}\left(c_{\theta_{g_i}}\right)
    \left|c_{\theta_{g_i}}\right|^{[\K_{g_i}:\F](1/2-s)}.
\end{equation}
Indeed, Theorem~\ref{Abhyankar} gives
$e(\K_{g_i}\mid\F)d(\theta_{g_i})=1$.
Moreover, since $c_{\theta_{g_i}}=g_i(\beta_{f_k})-\beta_v$,
the tower law for trace \cite[Proposition~5.48, p.~84]{MilneFT} gives
\begin{align*}
    \tr_{\K_{g_i}/\F}\left(c_{\theta_{g_i}}\right)
    &=m_{g_i}\tr_{g_i(\E_{f_k})/\F}
        \left(g_i(\beta_{f_k})\right)
      -km_{g_i}\tr_{\E_v/\F}(\beta_v)\\
    &=m_{g_i}\tr_{\E_{f_k}/\F}(\beta_{f_k})
      -km_{g_i}\tr_{\E_v/\F}(\beta_v)\\
    &=0.
\end{align*}

Recall the notation introduced in \eqref{CompositumDegree}.~Simplifying \eqref{GammaFactorTerm} further, we note that
\begin{equation}\label{AbsoluteValueTerm}
\left|c_{\theta_{g_{i}}}\right|^{[\K_{g_{i}} \, : \, \F](1/2-s)} = q_{\F}^{km_{g_{i}}(1/2-s)}
\end{equation}
and
\begin{equation}\label{CharacterTerm}
\theta_{g_{i}}^{-1}\left(c_{\theta_{g_{i}}}\right) =
\left({}^{g_{i}}\xi\right)^{-1}\left(N_{\K_{g_{i}}/g_{i}(\E_{f_{k}})} \left(c_{\theta_{g_{i}}}\right) \right) \eta^{-1}\left(N_{\K_{g_{i}}/\E_v}\left(c_{\theta_{g_{i}}}\right)\right).
\end{equation}
Furthermore, we have from comparing dimensions in Lemma~\ref{LemmaGalois} that
\[
\sum_i knm_{g_{i}}=kn\cdot n=kn^2.
\]
Dividing by $kn$, we get
\begin{equation}\label{DimensionIdentity}
\sum_i m_{g_{i}}=n.
\end{equation}
Equations \eqref{AbsoluteValueTerm} and \eqref{CharacterTerm} enable us to rewrite \eqref{GammaProduct} as
\begin{equation}\label{GammaNormProduct}
\begin{aligned}
& q_{\F}^{kn(1/2-s)} \prod_{i}\lambda_{\K_{g_{i}}/\F}(\psi_{\F})
\left({}^{g_{i}}\xi\right)^{-1}\left(N_{\K_{g_{i}}/g_{i}(\E_{f_{k}})}
\left(c_{\theta_{g_{i}}}\right)\right)\eta^{-1}\left(N_{\K_{g_{i}}/\E_v}\left(c_{\theta_{g_{i}}}\right)\right)\\
&=
q_{\F}^{kn(1/2-s)}\left(\prod_{i}\lambda_{\K_{g_{i}}/\F}(\psi_{\F})\right)\xi^{-1}\left(\prod_{i} g_{i}^{-1}\left(N_{\K_{g_{i}}/g_{i}(\E_{f_{k}})}\left(c_{\theta_{g_{i}}}\right)\right)\right)\\
& \cdot \eta^{-1}\left(\prod_{i} N_{\K_{g_{i}}/\E_v}\left(c_{\theta_{g_{i}}}\right)\right).
\end{aligned}
\end{equation}

By Theorem~\ref{Abhyankar}, $\K_{g_{i}}$ is finite and unramified over both $g_{i}(\E_{f_{k}})$ and $\E_v$; hence the extensions $\K_{g_{i}}/g_{i}(\E_{f_{k}})$ and $\K_{g_{i}}/\E_v$ are Galois.~Using \cite[Subsection 2.3, p.~11]{Adrian-Stevens} and \eqref{DimensionIdentity}, we can simplify \eqref{GammaNormProduct} as
\begin{equation}\label{GammaConjugateProduct}
\begin{aligned}
&q_{\F}^{kn(1/2-s)} \left(\prod_{i}\lambda_{\K_{g_{i}}/\F}(\psi_{\F})\right) \xi^{-1}\left(\prod_{i}
\prod_{\sigma\in\Gal(\K_{g_{i}}/g_{i}(\E_{f_{k}}))}\left(\beta_{f_{k}}-g_{i}^{-1}(\sigma(\beta_{v}))\right)\right)\\
&\cdot\eta^{-1}\left((-1)^{kn}\prod_{i}\prod_{\tau\in\Gal(\K_{g_{i}}/\E_v)}\left(\beta_{v}-\tau(g_{i}(\beta_{f_{k}}))\right)\right)\\
&=
q_{\F}^{kn(1/2-s)}\left(\prod_{i}\lambda_{\K_{g_{i}}/\F}(\psi_{\F})\right)\xi^{-1}\left(m_{\beta_{v}}(\beta_{f_{k}})\right)
\eta^{-1}\left((-1)^{kn}m_{\beta_{f_{k}}}(\beta_{v})\right)
\end{aligned}
\end{equation}
where $\Gal$ denotes the Galois group.~Using \eqref{LanglandsConstantTower}, \eqref{LanglandsConstantUnramified}, and \eqref{D(v)}, we can further rewrite \eqref{GammaConjugateProduct} as
\begin{equation}\label{GammaLanglandsProduct}
\begin{aligned}
&(-1)^{n-D(v)}q_{\F}^{kn(1/2-s)}\left(\prod_{i}\lambda_{g_{i}(\E_{f_{k}})/\F}(\psi_{\F})^{m_{g_{i}}}\right)\\
& \cdot \xi^{-1}\left(m_{\beta_{v}}(\beta_{f_{k}})\right) \eta^{-1}\left((-1)^{kn}m_{\beta_{f_{k}}}(\beta_{v})\right).
\end{aligned}
\end{equation}
From \cite[Subsection 28.7, Remark]{BH-7} and \eqref{LanglandsConstant}, we have that
\begin{equation}\label{LanglandsConstantConjugate}
\lambda_{g_{i}(\E_{f_{k}})/\F}(\psi_{\F})
=
\lambda_{\E_{f_{k}}/\F}(\psi_{\F}).
\end{equation}

Next, we have from \eqref{MinimalPolynomialSimple} and \eqref{MinimalPolynomialMiddle} that
\begin{equation}\label{PolynomialEvaluationSimple}
\begin{aligned}
m_{\beta_{v}}(\beta_{f_{k}})^{-1} &= \left(\beta_{f_{k}}^{n}-(v\varpi_{\F})^{-1}\right)^{-1} \\ &= \frac{\varpi_{\F}}{\sigma_{f_{k}}-v^{-1}} \\ &=
\beta_{f_{k}}^{-n} \frac{\sigma_{f_{k}}}{\sigma_{f_{k}}-v^{-1}}
\end{aligned}
\end{equation}
and
\begin{equation}\label{PolynomialEvaluationMiddle}
\begin{aligned}
m_{\beta_{f_{k}}}(\beta_{v})^{-1} &= \frac{\varpi_{\F}^{k}}
{v^{-k}-a_{k-1}v^{-(k-1)}-\cdots-a_1v^{-1}-a_0} \\ &= \frac{\varpi_{\F}^{k}}{f_{k}(v^{-1})}.
\end{aligned}
\end{equation}
Using \eqref{DimensionIdentity}, \eqref{LanglandsConstantConjugate}, \eqref{PolynomialEvaluationSimple}, \eqref{PolynomialEvaluationMiddle}, and the definition of $B_{f_{k}}$, we may simplify \eqref{GammaLanglandsProduct} as
\begin{align*}
&(-1)^{n-D(v)}q_{\F}^{kn(1/2-s)} \lambda_{\E_{f_{k}}/\F}(\psi_{\F})^n
\,\xi\left(\beta_{f_{k}}^{-n}\frac{\sigma_{f_{k}}}{\sigma_{f_{k}}-v^{-1}}\right)\eta\left(\frac{(-1)^{kn}\varpi_{\F}^{k}}{f_{k}(v^{-1})}\right)\\
& = (-1)^{n-D(v)}q_{\F}^{kn(1/2-s)}B_{f_{k}}^n \,\xi\left(\frac{\sigma_{f_{k}}}{\sigma_{f_{k}}-v^{-1}}\right)\eta\left(\frac{(-1)^{kn}\varpi_{\F}^{k}}{f_{k}(v^{-1})}\right).
\end{align*}
Lastly, \cite[Proposition 29.2]{BH-7} and \cite[p.~2]{Harris} tell us that
\[
\eta\left(\frac{(-1)^{kn}\varpi_{\F}^{k}}{f_{k}(v^{-1})}\right) =\varkappa_{\E_v/\F}\left(\frac{(-1)^{kn}\varpi_{\F}^{k}}{f_{k}(v^{-1})}\right)\omega_{(v, \, \varphi, \, \zeta')}\left(\frac{(-1)^{kn}\varpi_{\F}^{k}}{f_{k}(v^{-1})}\right)
\]
which completes the proof of Proposition~\ref{GaloisPropositionMain}.
\end{proof}

\section{Proof of Main Results}\label{MainProof}
In this section, we prove Theorems \ref{MainResult1} and \ref{MainResult2}.

\subsection{Distinguishing two $k$-middle supercuspidal representations} We first prove Theorem~\ref{MainResult1}.~Let $\pi_i=\pi_{(f_{k,i}, \, \chi_{k,i}, \, \zeta_{k,i})}$ and $\Phi_i=\Ind_{\E_{f_{k,i}}/\F}\xi_i$.
\begin{proof}[Proof of Theorem \ref{MainResult1}]
Let $\pi_1$ and $\pi_2$ be $k$-middle supercuspidal representations with the same central character and equal twisted gamma factors for all the twists in the hypotheses.~Since the local Langlands correspondence preserves these twisted gamma factors \cite{Harris}, we work with the corresponding Langlands parameters $\Phi_1$ and $\Phi_2$.~It suffices to show that their corresponding admissible pairs are conjugate.

Taking $\nu$ to be the trivial quasi-character in Proposition~\ref{PropositionGaloisTamelyRamified}, we obtain
\begin{equation}\label{SameB}
B_{f_{k,1}}=B_{f_{k,2}}.
\end{equation}
Comparing the same formula for an arbitrary tamely ramified quasi-character $\nu$ of $\F^\times$ then gives
\[
\nu\left(\frac{a_{0,2}}{a_{0,1}}\right)=1.
\]
Since this holds for every such $\nu$, we have $\frac{a_{0,2}}{a_{0,1}} \in 1+\mathcal{P}_{\F}$, and hence
\begin{equation}\label{SameConstantTerms}
\bar{a}_{0,1}=\bar{a}_{0,2}.
\end{equation}
    
Fix $v\in\mathcal{O}_{\F}^{\times}$.~For every character $\varphi$ of $k_{\F}^{\times}$, inflated to $\mathcal{O}_{\F}^{\times}$, we have
\begin{equation}\label{fStep1}
    \frac{\gamma(s, \, \Phi_{(f_{k,1}, \, \chi_{k,1}, \, \zeta_{k,1})}
\otimes\Phi_{(v, \, \varphi, \, \zeta')}, \, \psi_{\F})}{\gamma\left(s, \, \Phi_{(f_{k,1}, \, \chi_{k,1}, \, \zeta_{k,1})}
\otimes\Phi_{\left(v, \, \mathbbm{1}_{\mathcal{O}_{\F}^{\times}}, \, 1\right)}, \, \psi_{\F}\right)} =  \frac{\gamma(s, \, \Phi_{(f_{k,2}, \, \chi_{k,2}, \, \zeta_{k,2})}
\otimes\Phi_{(v, \, \varphi, \, \zeta')}, \, \psi_{\F})}{\gamma\left(s, \, \Phi_{(f_{k,2}, \, \chi_{k,2}, \, \zeta_{k,2})}
\otimes\Phi_{\left(v, \, \mathbbm{1}_{\mathcal{O}_{\F}^{\times}}, \, 1\right)}, \, \psi_{\F}\right)}
\end{equation}
where $\mathbbm{1}_{\mathcal{O}_{\F}^{\times}}$ is the trivial quasi-character of $\mathcal{O}_{\F}^{\times}$.~From Proposition \ref{GaloisPropositionMain} and \eqref{SameB}, we have that \eqref{fStep1} reduces to
\begin{equation}\label{fStep2}
    \varphi\left(\frac{f_{k,1}(v^{-1})}{f_{k,2}(v^{-1})}\right) = 1.
\end{equation}

Since \eqref{fStep2} holds for every such $\varphi$, we have
\[
\bar{f}_{k,1}(\bar{v}^{-1})
=
\bar{f}_{k,2}(\bar{v}^{-1})
\]
for any $v\in\mathcal{O}_{\F}^{\times}$.~Our assumption $q_{\F}>(k-1)^2$ implies that
\[
q_{\F}-1\geq(k-1)^2\geq k-1.
\]
We may therefore choose
$v_1,\ldots,v_{k-1}\in\mathcal{O}_{\F}^{\times}$
with distinct residues.~Let $w_r:=\bar{v}_r^{-1}$.~For each $r$, expanding
$\bar{f}_{k,1}(w_r)=\bar{f}_{k,2}(w_r)$ gives
\[
w_r^k-\sum_{j=1}^{k-1}\bar{a}_{j,1}w_r^j-\bar{a}_{0,1}
=
w_r^k-\sum_{j=1}^{k-1}\bar{a}_{j,2}w_r^j-\bar{a}_{0,2}.
\]
From \eqref{SameConstantTerms}, we have $\bar{a}_{0,1}=\bar{a}_{0,2}$.~Canceling these terms along with $w_r^k$ and dividing by $-w_r$ yields
\[
\sum_{j=1}^{k-1}\bar{a}_{j,1}w_r^{j-1}
=
\sum_{j=1}^{k-1}\bar{a}_{j,2}w_r^{j-1}.
\]

To illustrate how we solve this system, consider $k=4$.~The equations are
\[
\bar{a}_{3,1}w_r^2+\bar{a}_{2,1}w_r+\bar{a}_{1,1} = \bar{a}_{3,2}w_r^2+\bar{a}_{2,2}w_r+\bar{a}_{1,2}, \, \,  r=1,2,3.
\]
Subtracting the equation for $w_1$ from that for $w_2$ and dividing by $w_2-w_1$ gives
\[
\bar{a}_{3,1}(w_2+w_1)+\bar{a}_{2,1} = \bar{a}_{3,2}(w_2+w_1)+\bar{a}_{2,2}.
\]
Likewise, subtracting the equation for $w_1$ from that for $w_3$ and dividing by $w_3-w_1$ gives
\[
\bar{a}_{3,1}(w_3+w_1)+\bar{a}_{2,1} = \bar{a}_{3,2}(w_3+w_1)+\bar{a}_{2,2}.
\]
Subtracting these two equations, we obtain
\[
\bar{a}_{3,1}(w_3-w_2) = \bar{a}_{3,2}(w_3-w_2).
\]
Since $w_3\neq w_2$, we conclude that $\bar{a}_{3,1}=\bar{a}_{3,2}$.~Substituting back into either preceding equation gives $\bar{a}_{2,1}=\bar{a}_{2,2}$, and substituting into an original equation then gives $\bar{a}_{1,1}=\bar{a}_{1,2}$.

The same argument applies for general $k$.~The identity
\[
\frac{w^m-z^m}{w-z} = w^{m-1}+w^{m-2}z+\cdots+z^{m-1}
\]
shows that, with $z$ fixed, subtraction followed by division removes the constant term and replaces each monomial of degree $m$ by a monic polynomial of degree $m-1$.~Thus, starting with $k-1$ equations of degree at most $k-2$, each step leaves one fewer equation and reduces the degree bound by one.~All divisions are valid because the $w_r$ are distinct.~After $k-2$ steps, the remaining equation is
\[
\bar{a}_{k-1,1}=\bar{a}_{k-1,2}.
\]
Substituting back successively gives
\[
\bar{a}_{j,1}=\bar{a}_{j,2}, \, \,  1\leq j\leq k-1.
\]
Together with the equality of constant terms, this proves that
\[
\bar{f}_{k,1}=\bar{f}_{k,2}.
\]

By Proposition~\ref{Bijection}, we may realize both representations using a common lift $f_{k}$ of $\bar{f}_{k,1}=\bar{f}_{k,2}$.~Hence,
\[
\E_{f_{k, 1}}=\E_{f_{k, 2}}=\E_{f_{k}}, 
\]
\[
\beta_{f_{k, 1}}=\beta_{f_{k, 2}}=\beta_{f_{k}},
\]
and
\[
\sigma_{f_{k, 1}}=\sigma_{f_{k, 2}}=\sigma_{f_{k}}.
\]
The equality of the $B_{f_{k,i}}$ now gives $\xi_1(\beta_{f_{k}})=\xi_2(\beta_{f_{k}})$.~Set $\delta = \xi_1\xi_2^{-1}$.~Proposition~\ref{1units} and the equal-central-character condition \eqref{detidentity} give
\[
\delta \, \big|_{\, 1+\mathcal P_{\E_{f_{k}}}}=1,\, \, \delta(\beta_{f_{k}}) = 1, \, \, 
\delta \, \big|_{\, \F^{\times}}=1.
\]
Since $\sigma_{f_{k}}=\beta_{f_{k}}^n\varpi_{\F}$, we also have $\delta(\sigma_{f_{k}})=1$.

Using the equality of twisted gamma factors for twists by simple supercuspidal representations of $\GL(n, \, \F)$ and canceling the remaining common factors in \eqref{simpleformula}, we obtain
\[
\delta(\sigma_{f_{k}}-v^{-1})=1.
\]
Thus, the character $\bar\delta$ induced on $\k_{\L_{f_{k}}}^{\times}$ satisfies
\begin{equation}\label{linevalues}
\bar\delta(\bar\sigma_{f_{k}}-t)=1
\end{equation}
for all $t \in \k_{\F}$.

If $\bar\delta$ were non-trivial, \eqref{linevalues} and Theorem~\ref{KatzEstimate} would imply
\[
q_{\F}\leq(k-1)\sqrt{q_{\F}},
\]
contrary to our assumption $q_{\F}>(k-1)^2$.~Hence $\bar\delta=1$.~Since
\[
\E_{f_{k}}^\times=\langle\beta_{f_{k}}\rangle\times\mu'_{\L_{f_{k}}}\times(1+\mathcal P_{\E_{f_{k}}}),
\]
we have $\delta=1$.~Thus, $\xi_1=\xi_2$, so $\Phi_1\cong\Phi_2$.~The local Langlands correspondence gives $\pi_1\cong\pi_2$, proving Theorem~\ref{MainResult1}.
\end{proof}

\subsection{Alternate supercuspidal representations of depth $\frac{1}{n}$}

In this subsection, we use twisted gamma factors to distinguish $k$-middle supercuspidals from other depth $\frac{1}{n}$ supercuspidal representations of $\GL(kn, \, \F)$.~We use the conductor formulas in \cite{BHK} and the notion of \textit{completely distinct} supercuspidal representations, as in \cite[Section 4]{Luo}.

Recall from Subsection~\ref{MaximalSimpleTypes} that $\varphi_\rho$ denotes the minimal polynomial of a simple stratum associated with a supercuspidal representation $\rho$ and is an invariant of $\rho$.~Furthermore, we use the local Langlands correspondence, together with Theorem~\ref{TheoremGalois}, to pass between a supercuspidal representation $\rho$ and its corresponding admissible pair, considered up to $\F$-isomorphism.

\begin{definition}
Let $\pi_1$ and $\pi_2$ be positive-depth supercuspidal representations of $\GL(N_{1}, \, \F)$ and $\GL(N_{2}, \, \F)$, respectively.~We say that the $\pi_i$ are \textit{completely distinct} if either $d(\pi_1)\ne d(\pi_2)$ or $d(\pi_1)=d(\pi_2)$ but $\varphi_{\pi_1}\ne\varphi_{\pi_2}$.
\end{definition}

We will use the following consequence of \cite[Theorem 6.5(ii),(iii)]{BHK}, also used in \cite[Proposition 4.3]{Luo}.~If $\pi_{1}$ and $\pi_{2}$ are supercuspidal representations of $\GL(N_{1}, \, \F)$ and $\GL(N_{2}, \, \F)$ with $N_{1}\ne N_{2}$ and the same positive depth $a$, then
\begin{equation}\label{BHKinequality}
f(\pi_{1}\times\pi_{2}^\vee) \leq N_{1}N_{2}(1+a),
\end{equation}
where $\pi_{2}^\vee$ denotes the smooth contragredient of $\pi_{2}$.~Equality holds when $\pi_{1}$ and $\pi_{2}$ are completely distinct, and the inequality is strict when they are not completely distinct.

\begin{lemma}\label{degreelemma}
Let $\rho$ be a supercuspidal representation of $\GL(kn, \, \F)$ with depth $\frac{1}{n}$ which is not a $k$-middle supercuspidal.~Then $d := \deg\varphi_\rho$ is a positive proper divisor of $k$.
\end{lemma}

\begin{proof}
Choose a maximal simple type associated with a simple stratum $[\mathfrak A,m,0,\beta]$, and put $\E=\F[\beta]$.~The depth condition and maximality give $e\left(\E | \F\right)=mn$.~Furthermore, we have that 
\begin{equation}\label{Tower1Lemma}
[\E : \F] = e\left(\E | \F\right)f\left(\E | \F\right) = mnf\left(\E | \F\right) \mid kn.
\end{equation}

Let
\[
Y_\beta := \beta^n\varpi_{\F}\in\mathcal{O}_{\E}^\times.
\]
Since $d=[\k_{\F}(\overline{Y_\beta}):\k_{\F}]$, the tower law gives
\[
f\left(\E | \F\right) = \left[\k_{\E}:\k_{\F}\left(\overline{Y_\beta}\right)\right]\,d.
\]
Thus, $d\mid f\left(\E | \F\right)$, and \eqref{Tower1Lemma} implies $d\mid k$.

If $d=k$, then $m = 1$ and $f\left(\E | \F\right) = k$.~Hence, the corresponding level-one stratum is minimax, with parameter field of ramification index $n$ and residue degree $k$ \cite[(1.4.14)--(1.4.15)]{BK}.~In this case, $\rho$ would be a $k$-middle supercuspidal representation, contrary to our assumption.~Thus, $d$ is a proper divisor of $k$.
\end{proof}

\begin{proposition}\label{separation}
Let $\pi$ be a $k$-middle supercuspidal and $\rho$ a supercuspidal representation of $\GL(kn, \, \F)$ of depth $\frac{1}{n}$ that is not a $k$-middle supercuspidal.~Then there exists a proper divisor $d$ of $k$ and a $d$-middle supercuspidal $\tau$ of $\GL(dn, \, \F)$ such that
\[
\gamma(s, \, \pi\times\tau^\vee, \, \psi_{\F}) \neq \gamma(s, \, \rho\times\tau^\vee, \, \psi_{\F}).
\]
\end{proposition}

\begin{proof}
Let $d :=\deg\varphi_\rho$.~By Lemma~\ref{degreelemma}, $d$ is a positive proper divisor of $k$.~Furthermore, the construction in Subsection~\ref{Middle} enables us to choose a $d$-middle supercuspidal representation $\tau$ of $\GL(dn, \, \F)$ with $\varphi_\tau=\varphi_\rho$ and $d(\tau)=\frac{1}{n}$; for $d=1$, we use the simple supercuspidal construction given in Subsection \ref{SimpleSupercuspidalRepresentations}.

Then $\rho$ and $\tau$ are not completely distinct, whereas $\pi$ and $\tau$ are completely distinct because their minimal polynomials have different degrees.~Hence, \eqref{BHKinequality} gives
\[
f(\rho\times\tau^\vee) < kdn^2\left(1+\frac{1}{n}\right) = f(\pi\times\tau^\vee).
\]
Theorem \ref{TheoremConductor} now implies
\[
\gamma(s, \, \rho\times\tau^\vee, \, \psi_{\F}) \neq \gamma(s, \, \pi \times\tau^\vee, \, \psi_{\F}),
\]
proving the proposition.
\end{proof}

\begin{remark}
\normalfont The contragredient of a $d$-middle supercuspidal is again a $d$-middle supercuspidal \cite[Subsection 3.5]{BH-7}.
\end{remark}

\subsection{Proof of Theorem~\ref{MainResult2}}

\begin{proof}[Proof of Theorem \ref{MainResult2}]
Let $\pi$ and $\pi'$ satisfy the hypotheses of Theorem~\ref{MainResult2}.~Taking the trivial quasi-character as a twist gives equality of the standard gamma factors and hence $d(\pi)=d(\pi')$ \cite[proof of Proposition 2.1(iii)]{Xu}.~Theorem \ref{MainResult2} now follows from Theorem \ref{MainResult1} and Proposition \ref{separation}.
\end{proof}

\end{document}